\documentclass[12pt]{article}
\usepackage[T1]{fontenc}
\usepackage{verbatim}
\usepackage{amsmath,amsthm}
\usepackage{xspace}
\usepackage{pifont}
\usepackage{graphicx}
\usepackage{amssymb}
\usepackage{epic, eepic}
\usepackage{dsfont}
\usepackage{amssymb}
\usepackage{makeidx}
\usepackage{mathrsfs}
\usepackage{xcolor, colortbl}
\usepackage[numbers]{natbib}
\usepackage{enumerate}
\usepackage{afterpage}
\usepackage{mathtools}
\usepackage{overpic} 
\usepackage{bm}
\usepackage{bbm}
\usepackage{subcaption}
\usepackage{epsf}
\usepackage{graphics,color}
\usepackage[export]{adjustbox}
\usepackage{tcolorbox}
\usepackage{authblk}
\usepackage[normalem]{ulem}
\RequirePackage[colorlinks,citecolor=blue,urlcolor=blue]{hyperref}
\mathtoolsset{showonlyrefs}

\def\eps{\varepsilon}

\def\wt{\widetilde}

\def\E{{\mathbb{E}}}
\def\rr{\rightarrow}
\def\one{\mathds{1}}
\newcommand{\R}{{\mathbb R}}
\newcommand{\C}{{\mathbb C}}

\newcommand{\N}{{\mathbb N}}

\DeclareMathOperator*{\argmin}{arg\,min}

\newcommand{\bd}{\begin{displaymath}}
\newcommand{\ed}{\end{displaymath}}
\newcommand{\be}{\begin{equation}}
\newcommand{\ee}{\end{equation}}
\newcommand{\bq}{\begin{eqnarray}}
\newcommand{\eq}{\end{eqnarray}}
\newcommand{\bn}{\begin{eqnarray*}}
\newcommand{\en}{\end{eqnarray*}}
\newcommand{\dl}{\delta}

\newtheorem{theorem}{Theorem}[section]
\newtheorem{lemma}[theorem]{Lemma}

\newtheorem{remark}[theorem]{Remark}

\numberwithin{equation}{section}

\title{
The Mercer-Young Theorem for Matrix-Valued Kernels on Separable Metric Spaces}

\author[]{Eyal Neuman}
\author[]{Sturmius Tuschmann\thanks{ST is supported by the EPSRC Centre for Doctoral Training in Mathematics of Random \mbox{Systems}: Analysis, Modelling and Simulation (EP/S023925/1)}}
\affil[]{Department of Mathematics, Imperial College London}

\begin{document}
\maketitle

\begin{abstract} 
\noindent We generalize the characterization theorem going back to Mercer and Young, which states that a symmetric and continuous kernel is positive definite if and only if it is integrally positive definite, to matrix-valued kernels on separable metric spaces. We also demonstrate the applications of the generalized theorem to the field of convex optimization and other areas.  \vspace{0.5cm}
\end{abstract} 

\begin{description}
\item[Mathematics Subject Classification (2010/2020):] 	28A25, 43A35, 47B34
\item[Keywords:]  positive definite kernel, matrix-valued kernel, measure theory, operator theory, convex optimization
\end{description}

\section{Introduction and Main Results}
Positive definite kernels are a generalization of positive definite functions, and were first introduced by Mercer in his seminal paper \cite{mercer1909functions} from 1909. Namely, motivated by Hilbert's earlier work on Fredholm integral equations \cite{hilbert1904grundzuge}, Mercer's original goal was the characterization of all symmetric and continuous kernels $K:[a,b]\times [a,b]\to\R$ on a compact interval $[a,b]\subset\R$ which are strictly integrally positive definite with respect to the space $C([a,b],\R)$ of continuous functions on $[a,b]$, i.e., which satisfy
\be
\int_a^b\int_a^b f(x)K(x,y)f(y)dxdy>0,\textrm{ for all } f\in C([a,b],\R)\textrm{ with } f\neq 0.
\ee
However, as mentioned by Mercer in the introduction of \cite{mercer1909functions}, he quickly figured out that this condition going back to  Hilbert (see Chapter V of \cite{hilbert1904grundzuge}) was too restrictive to allow a characterization by an equivalent discrete condition on the kernel $K$. Thus, he investigated kernels $K$ which are integrally positive definite with respect to $C([a,b],\R)$,  i.e., which satisfy
\be
\int_a^b\int_a^b f(x)K(x,y)f(y)dxdy\geq 0,\textrm{ for all } f\in C([a,b],\R).
\ee
He then showed that a symmetric and continuous kernel $K:[a,b]\times [a,b]\to\R$ is integrally positive definite with respect to $C([a,b],\R)$ if and only if $K$ is a positive definite kernel, that is, $K$ satisfies
\be
    \sum_{i,j=1}^n c_i K(x_i,x_j)c_j\geq 0, \text{ for any } n\in\N, \ x_1,\dots,x_n\in[a,b],\  c_1,\dots,c_n\in\R.
\ee
Here, positive definite kernels, as they are known today, appeared for the first time in the literature, where the term 'positive definite' is now well-established in spite of the weak inequality in the condition. Just one year later, Mercer's characterization theorem was generalized by Young \cite{young1910note} to the larger function space $L^1([a,b],\R)$ of integrable functions on $[a,b]$. The results of Mercer and Young are summarized in the following theorem (see \S 10 in \cite{mercer1909functions} and \S 1 in \cite{young1910note}).
\begin{theorem}\textup{\textbf{(Mercer-Young, 1909-1910)}} \label{thm:mercer-young}
Let $[a,b]\subset\R$ be a compact interval and $K:[a,b]\times[a,b]\to\R$ be a symmetric and continuous kernel. Then the following statements are equivalent:
\begin{enumerate}[(i)]
\item $K$ is positive definite, i.e.,
    \be
    \sum_{i,j=1}^n c_i K(x_i,x_j)c_j\geq 0,\text{ for any } n\in\N, \ x_1,\dots,x_n\in[a,b],\  c_1,\dots,c_n\in\R.
    \ee
\item $K$ is integrally positive definite with respect to $C([a,b],\R)$, i.e.,
    \be
    \int_a^b\int_a^b f(x)K(x,y)f(y)dxdy\geq 0, \text{ for all }f\in C([a,b],\R).
    \ee
\item $K$ is integrally positive definite with respect to $L^1([a,b],\R)$, i.e.,
    \be
    \int_a^b\int_a^b f(x)K(x,y)f(y)dxdy\geq 0, \text{ for all }f\in L^1([a,b],\R).
    \ee
\end{enumerate}
\end{theorem} 
The definition of a positive definite kernel may be generalized in various ways, most importantly, to more general domains and codomains. The object of this work is the generalization of Theorem \ref{thm:mercer-young} to matrix-valued kernels on separable metric spaces.  Our main result is the following: 
\begin{theorem} \label{thm:generalization}
Let $X$ be a separable metric space equipped with a locally finite measure $\mu$ on $\mathcal{B}(X)$ that has support $X$. Let $N\in\N$ and $K:X\times X\to\R^{N\times N}$ be a bounded and continuous kernel satisfying $K(x,y)=K(y,x)^\top$ for all $x,y\in X$. Then the following statements are equivalent:
\begin{enumerate}[(i)]
\item $K$ is positive definite, i.e.,
    \be \label{eq:positive definite kernel real}
    \sum_{i,j=1}^n c_i^\top K(x_i,x_j)c_j\geq 0, \text{ for any } n\in\N, \ x_1,\dots,x_n\in X,\  c_1,\dots,c_n\in\R^N. 
    \ee
\item $K$ is integrally positive definite with respect to $C(X,\R^N)\cap L^1(X,\R^N)$, i.e.,
   \be\label{eq:ipd wrt C}
    \int_X\int_X f(x)^\top K(x,y)f(y)d\mu (x)d\mu (y)\geq 0, \text{ for all }f\in C(X,\R^N)\cap L^1(X,\R^N).\vspace{-2mm}
    \ee
\item $K$ is integrally positive definite with respect to $L^1(X,\R^N)$, i.e.,
    \be\label{eq:ipd wrt L1}
    \int_X\int_X f(x)^\top K(x,y)f(y)d\mu (x)d\mu (y)\geq 0, \text{ for all }f\in L^1(X,\R^N). 
    \ee
\end{enumerate}
\end{theorem} 
\begin{remark}
Theorem \ref{thm:generalization} simplifies to Theorem \ref{thm:mercer-young} when choosing $X=[a,b]$, $N=1$, and $\mu$ to be the Lebesgue measure on $\mathcal{B}([a,b])$, as all continuous functions on the compact interval $[a,b]$ are bounded and thus Lebesgue-integrable.
\end{remark}
In Section \ref{sec:proof}, we prove Theorem \ref{thm:generalization} by establishing the chain of implications $$(iii)\Rightarrow (ii)\Rightarrow(i)\Rightarrow (iii),$$ with the first one being trivial.
For the second implication $(ii)\Rightarrow(i)$, we generalize Mercer's original geometric argument (see Part II of \cite{mercer1909functions}) from compact intervals to separable metric spaces. A significant challenge here is that the original proof explicitly constructs continuous cutoff functions, a method that does not directly extend to general metric spaces. We overcome this obstacle by employing topological tools, notably Urysohn’s lemma, to construct these functions implicitly. Moreover, an advantage of this adaptation is that it extends naturally from scalar-valued to matrix-valued kernels, requiring only the incorporation of an appropriate matrix norm. For the third implication $(i)\Rightarrow(iii)$, the extension to the matrix-valued framework is considerably more difficult since most established proofs are based on spectral theorems for scalar kernels. We address this difficulty by applying a version of Mercer’s theorem for matrix-valued kernels on separable metric spaces developed by \citet{devito2013extension}. In our proof, we first assume a finite measure and subsequently extend our findings to the locally finite setting.

\paragraph{Related literature.} Although the classical Mercer-Young Theorem (Theorem \ref{thm:mercer-young}) dates back to  the early twentieth century \cite{mercer1909functions,young1910note}, various related results have emerged in the literature since then. Inspired by \citet{caratheodory1907variabilitatsbereich} and \citet{toeplitz1911fourier},   
\citet{mathias1923positive} and \citet{bochner1932vorlesungen,bochner1933monotone} studied (independently of Mercer’s work) positive definite functions $\psi:\R\to\C$, which constitute a special class of positive definite kernels $K$ of the form $K(x,y)=\psi(x-y)$.
In particular, Bochner showed the equivalence of (i) and (ii) for bounded, continuous positive definite functions (see \S 20 in \cite{bochner1932vorlesungen}). In the same work, he also proved his celebrated result known as Bochner's theorem, which characterizes continuous positive-definite functions as exactly those that can be represented as the Fourier transform of a finite positive measure (see Theorem 23 in \S 20 of \cite{bochner1932vorlesungen} and Chapter 1.4 of \cite{rudin2017fourier} for a more general version). Motivated by Bochner's results, \citet{cooper1960positive} studied integrally positive definite functions and derived the equivalence of (i) and (iii) for bounded, continuous positive definite functions. Since their introduction, Bochner's theorem and the theory of positive definite functions have been extended into several directions, see \cite{bhatia2009positive,glockner2003positive,stewart1976positive} for an overview. 

Although the aforementioned theory of positive definite functions was not directly influenced by Mercer, another branch of literature on the more general concept of positive definite kernels was inspired by Mercer’s theorem, which provides a representation of a positive definite kernel as a sum of products of its eigenvalues and eigenfunctions (see \S 10 in \cite{mercer1909functions}). For instance, Mercer’s work \cite{mercer1909functions} led to the development of the theory of reproducing kernel Hilbert spaces by \citet{moore1916properly,moore1935general}, \citet{aronszajn1950theory}, and others. More recently, among other extensions, Mercer's theorem was generalized to kernels on more general domains \cite{ferreira2009eigenvalues,steinwart2012mercer,sun2005mercer} and to kernels on more general codomains \cite{devito2013extension}. In particular, when $K$ is a scalar-valued kernel on a general domain $X$, \citet{ferreira2009eigenvalues} showed in an $L^2$-setting the implication $(i)\Rightarrow (iii)$ (which in turn yields $(i)\Rightarrow (ii)$) under the assumptions that $X$ is a topological space and $\mu$ is strictly positive (see Theorem 2.3 in \cite{ferreira2009eigenvalues}). Moreover, they proved that $(iii)\Rightarrow (i)$ under the assumptions that $X$ is a locally compact Hausdorff space and $\mu$ is a Radon measure (see Corollary 2.2 in \cite{ferreira2009eigenvalues}). For comprehensive surveys on positive definite kernels and Mercer's theorem, see \cite{buhmann2000radial,fasshauer2011positive,stewart1976positive}. 

\paragraph{Our contribution.} In this work, we establish the Mercer-Young Theorem for matrix-valued kernels on separable metric spaces (Theorem \ref{thm:generalization}). Our contributions can be summarized as follows:
\begin{itemize} 
    \item {General domains:} we prove the implication $(ii)\Rightarrow (i)$ for kernels on general domains, thereby closing a gap in the literature. To achieve this, we extend Mercer's original proof (see Part II of \cite{mercer1909functions}) from compact intervals to separable metric spaces. This extension is not straightforward, since the original proof exploits the structure of the domain to construct auxiliary cutoff functions explicitly. We resolve this issue by means of topological methods, particularly Urysohn’s lemma, to construct those auxiliary functions implicitly, and thereby leverage the topological regularity of the general domain (see equation \eqref{eq:cutoff function} and Figures \ref{figure:functions}-\ref{figure:square}).
     \item {General codomains:} we establish the equivalence of the statements $(i)$, $(ii)$, and $(iii)$ for matrix-valued kernels, a topic that has not been addressed even for simple domains, despite its importance for various emerging applications. In particular, to derive the implication $(i) \Rightarrow (iii)$, we provide a novel approach based on a version of Mercer’s theorem for matrix-valued kernels due to \citet{devito2013extension}. Namely, we combine the spectral decomposition of the general kernel with measure theoretic arguments to first prove the implication for finite measures (Case~1), and then generalize our result to locally finite measures (Case~2).  
\end{itemize}

\paragraph{Applications of Theorem \ref{thm:generalization}.} 
Since their introduction more than 100 years ago, positive definite kernels and their generalizations have found numerous applications in the fields of functional, numerical and stochastic analysis, approximation theory, computer graphics, engineering design, fluid dynamics, geostatistics, probability theory, statistics and  machine learning (see \cite{fasshauer2011positive} for a survey). Therefore, the result of Theorem \ref{thm:generalization} is not only of theoretical interest, but also of relevance for applications in the aforementioned fields. In particular, the minimization of the energy integral 
$$
\mathcal E_\mu(A)=  \int_{\mathbb R}\int_{\mathbb R} K(x-y) d\mu(x) d\mu(y), 
$$
over all probability measures $\mu$ supported on a fixed compact set $A \subset \R$ plays an important role in potential theory. In the case that a minimizing measure $\mu^*$ exists, it is called an equilibrium measure. The capacity of the set $A$ is then given by 
\be \label{enr}
\textrm{cap}(A) = \frac{1}{\mathcal E_{\mu^*}(A)}. 
\ee
A discrete version of the energy integral is defined as follows,  
\be\label{enr2}
 E_{\omega_N}(A)=   \frac{1}{N^2}\sum_{ x,y \in \omega_N} K(x-y), 
\ee
with the infimum of $E_{\omega_N}(A)$ taken over all $N$-point configurations $\omega_N=\{x_1,\ldots,x_N\}$ in $A$. The  special case of the Riesz kernel, i.e., $K(x,y) = \frac{1}{|x-y|^s}$ with $s>0$, has been extensively studied, and classical results prove that the empirical measure of the minimizer of \eqref{enr2}, $\omega^*_N =\{x^N_1,...,x_N^N\}$, given by $\frac{1}{N}\sum_{i=1}^N\delta_{x_i^N}$ converges to $\mu^*$ in the weak* topology as $N\rr \infty$, when $\mu^*$ is of finite energy. Results for more general kernels and for measure spaces supported on manifolds are subject to current research in potential theory (see \cite{Bjorck56,Fuglede60,Hardin04} and the comprehensive textbooks \cite{Borodachov19,Helms09}).

In recent years, another variant of this problem has gained significant interest in the areas of convex optimization and control theory, see \citep{abi2021integral, abijaber2023equilibrium,Alf-Schied-13, GSS,wang2022linear} and references therein.   
In this class of optimization problems the cost functionals consist of quadratic terms containing kernels along with lower-order terms, i.e., they are of the form 
  \be \label{v1}
V^1 := \min_{u \in \mathcal A}\, \E\left[   \int_{\mathbb R_{+}}\int_{\mathbb R_{+}}u(t)^\top K(t,s)u(s)ds dt  + ...\right],
    \ee
where optimization is performed over a class $\mathcal A$ of admissible stochastic or deterministic controls in $L^1(\R_+,\R^N)$, and the expectation is omitted in the deterministic case. Results such as Theorem \ref{thm:generalization} are of central importance for approximating these problems and reducing their dimension. In particular, one may consider the following minimization problem,  
  \be  \label{v2}
V^2 := \inf_{ \mathcal P \in \mathbb{P}} \min_{ v\in  \Xi(\mathcal P)}\E\left[   \sum_{ i,j \geq 1}  v_i^\top K(t_i,t_j)v_j   + ...\right],
    \ee
where $\mathbb{P}$ is the class of all locally finite 
partitions $\mathcal{P}$ of $\mathbb{R}_{+}$ and $\Xi(\mathcal P)$ is the class of piecewise constant stochastic or deterministic controls in $\mathcal{A}$ varying only on the partition $\mathcal P \in \mathbb{P}$. Then \eqref{v2} provides an approximation to the stochastic or deterministic control problem \eqref{v1} by reducing it to a finite-dimensional matrix optimization problem (see Theorem 3 in \cite{alfonsi2016multivariate} for a special case). In this setting, conditions \eqref{eq:positive definite kernel real} and \eqref{eq:ipd wrt C} are essential for characterizing the convexity of the cost functionals in \eqref{v2} and \eqref{v1}, respectively (see e.g \citep{abijaber2023equilibrium,alfonsi2016multivariate, GSS}). Indeed one can easily notice that if there exists a deterministic strategy $u\in \mathcal A$ such that \eqref{eq:ipd wrt C} doesn't hold, i.e.,  
$$
  \int_{\mathbb R_{+}}\int_{\mathbb R_{+}}u(t)^\top K(t,s)u(s)ds dt <0, 
$$
then for any constant $c>1$, the performance of $cu \in \mathcal{A}$ will be even lower, hence the value of \eqref{v1} is $-\infty$ and the optimization problem is ill-posed.

The extension of Theorem \ref{thm:mercer-young} to separable metric spaces is of particular interest when dealing with kernel methods in machine learning, where positive definite kernels defined on general input spaces are often employed. We also refer to \citep{brouard2016input,caponnetto2008universal,kadri2016operator,minh2016unifying,saha2020learning}, where the positive definite property of a matrix-valued kernel is associated with its reproducing property and applications in multi-task learning and related areas are studied. In this class of machine learning problems the associated loss functions are often quadratic and can be compared to \eqref{v1}, with general input spaces as domains of integration.
Similar optimization problems arise in the scenario of kernel estimation. In this setup the kernels describe a linear persistent temporal interaction between the input and the output. 
Specifically, by considering a time grid over an interval $[0,T]$, that is, $0=t_1<...<t_M=T$ for some $M \in \mathbb{N}$, the observed $N$ input and output data points $(u^{(n)},y^{(n)}) \in \mathbb{R}^M\times \mathbb{R}^M$ satisfy, 
\begin{equation}
\label{eq:price_impact_linear_regression}
y^{(n)} =K^\star u^{(n)} +\eps^{(n)}, 
\quad \text{with}\quad \eps^{(n)}= (\eps^{(n)}_{t_i})_{i=1}^{M}, \quad n=1,...,N,  
\end{equation}
where $(K^\star_{i,j})_{1\le j\le i\le M}$ are the unknown kernel coefficients and $(\eps^{(n)})_{n=1}^N$ are conditionally independent random variables representing the noise in the system. 
The kernel can be estimated via  a least-squares method, i.e., by minimizing  the following quadratic loss over a certain class of admissible kernels $\mathscr{K}_{\textrm{ad}} $,  
\begin{equation} \label{eq:G_n_lambda_volterra}
K_{N,\lambda}\coloneqq\argmin_{K \in \mathscr{K}_{\textrm{ad}}}
\left(
\sum_{n=1}^N \|y^{(n)}-K u^{(n)}\|^2+\lambda\|K\|_{\mathbb{R}^{M \times M}}^2
\right),
\end{equation}
where $\lambda>0$ is a given regularization parameter. While solving this class of problems for a discrete time grid is attainable, the solution for continuous time is much more challenging. Hence, Theorem \ref{thm:generalization} can help to achieve approximation results for the continuous-time case by using projections to discrete time (see \citep{BENATIA2017269,zhang_24,Wolf_24}). 

Finally, positive definite matrix-valued kernels play a crucial role in the numerical analysis of partial differential equations (PDEs), particularly in kernel-based discretization methods. As discussed in \cite{giesl2018kernel}, such kernels are often used for solving matrix-valued PDEs, also see \cite{giesl2021matrix,giesl2021computation}. Matrix-valued PDEs commonly arise in areas like dynamical systems, where contraction metrics are needed to study stability and invariant sets. The use of positive definite matrix-valued kernels ensures that numerical approximations preserve important structural properties such as stability and convergence. Since Theorem \ref{thm:generalization} characterizes such positive definite matrix-valued kernels, it helps to determine whether a given matrix-valued kernel admits a reproducing kernel Hilbert space representation and ensures that numerical discretization schemes based on such kernels retain their fundamental properties. Thus, it strengthens the foundation of kernel-based numerical PDE methods by aiding in the selection of  kernels.

\section{Proof of the Main Theorem}\label{sec:proof}
This section is dedicated to the proof of Theorem \ref{thm:generalization}, which will be given by showing the following three implications,
\be
\eqref{eq:ipd wrt L1}\Rightarrow \eqref{eq:ipd wrt C},\quad\eqref{eq:ipd wrt C}\Rightarrow \eqref{eq:positive definite kernel real},\quad \eqref{eq:positive definite kernel real}\Rightarrow \eqref{eq:ipd wrt L1}.
\ee
In order to prove the third implication, the following auxiliary lemma will be needed:
\begin{lemma}\label{lemma:equivalence of real and complex}
Let $X$ be an arbitrary set and $K :X\times X\to \R^{N\times N}$ be a mapping with $K(x,y)=K(y,x)^\top$ for all $x,y\in X$. Then condition \eqref{eq:positive definite kernel real} is equivalent to
\be
\sum_{i,j=1}^n \overline{z_i}^\top K(x_i,x_j)z_j\geq 0,  \ \text{for any }n\in\N,\  x_1,\dots, x_n\in X,\  z_1,\dots,z_n\in\C^N.  \label{eq:positive definite kernel complex}
\ee
\end{lemma}
\begin{proof}[Proof of Lemma \ref{lemma:equivalence of real and complex}]
Let $n\in\N$, $x_1,\dots, x_n\in X$, $z_1,\dots,z_n\in\C^N$, and assume that condition \eqref{eq:positive definite kernel real} holds. Then a direct computation using \eqref{eq:positive definite kernel real} yields,
\begin{align}
&\sum_{i,j=1}^n \overline{z_i}^\top K(x_i,x_j)z_j
=\sum_{i,j=1}^n ( \operatorname{Re}(z_i)-i\operatorname{Im}(z_i))^\top  K(x_i,x_j)(\operatorname{Re}(z_j)+i\operatorname{Im}(z_j))  \\
&\geq -i\sum_{i,j=1}^n \operatorname{Im}(z_i)^\top  K(x_i,x_j)\operatorname{Re}(z_j) 
+ i\sum_{i,j=1}^n \operatorname{Re}(z_i)^\top  K(x_i,x_j)\operatorname{Im}(z_j)  \\
&=-i\sum_{i,j=1}^n \operatorname{Im}(z_i)^\top  K(x_i,x_j)\operatorname{Re}(z_j) 
+ i\sum_{i,j=1}^n \operatorname{Im}(z_j)^\top  K(x_i,x_j)^\top\operatorname{Re}(z_i)=0,  
\end{align}
since $K(x_i,x_j)^\top=K(x_j,x_i)$ by assumption. The reverse direction is trivial.
\end{proof}

\begin{proof}[Proof of Theorem \ref{thm:generalization}] We proceed by showing the three implications,
\be 
\eqref{eq:ipd wrt L1}\Rightarrow \eqref{eq:ipd wrt C},\quad  \eqref{eq:ipd wrt C}\Rightarrow \eqref{eq:positive definite kernel real},  \quad  \eqref{eq:positive definite kernel real}\Rightarrow \eqref{eq:ipd wrt L1}.
\ee
The first implication \eqref{eq:ipd wrt L1}$\,\Rightarrow\,$\eqref{eq:ipd wrt C} is trivial.

In order to prove \eqref{eq:ipd wrt C}$\,\Rightarrow\,$\eqref{eq:positive definite kernel real}, we will generalize Mercer's original proof (see Part II of \cite{mercer1909functions}). Assume that condition \eqref{eq:ipd wrt C} holds. Fix $n\in\N$ and $n$ points $x_1,\dots,x_n\in X$. Without loss of generality, we can assume the $n$ points to be distinct, because for $n\geq 2$, $x_{n-1}=x_n$ and any $c_1,\dots,c_n\in\R^N$, we can define $\wt{c}_{n-1}\vcentcolon=c_{n-1}+c_n$ and $\wt{c_i}\vcentcolon=c_i$ for $1\leq i\leq n-3$ so that the following holds,
\be
\sum_{i,j=1}^n c_i^\top K(x_i,x_j)c_j=\sum_{i,j=1}^{n-1} \wt{c_i}^\top K(x_i,x_j)\wt{c_j}.
\ee
This assumption allows us to choose real numbers $\dl,\eps>0$ sufficiently small such that
\be\label{eq:disjoint balls}
B_{\dl+\eps}(x_i)\cap B_{\dl+\eps}(x_j)=\emptyset,\quad\textrm{for all }i,j=1,\dots ,n\text{ with } i\neq j.
\ee
Since $\mu$ is locally finite, we can moreover assume that 
\be \label{eq:finite measure}
\mu(\overline{B_{\dl+\eps}(x_i)})<\infty,\quad\textrm{for all }i=1,\dots ,n.
\ee 
Here $B_r(x)\subset X$ denotes the open ball of radius $r>0$ centred at $x\in X$. As a metric space, $X$ is in particular a normal topological space. Therefore, using Urysohn's Lemma (see Lemma 15.6 in \cite{willard2012general}), for $i=1,\dots, n$, there exist continuous functions illustrated in Figure \ref{figure:functions} given by,
\be\label{eq:cutoff function}
f_{x_i,\dl,\eps}:X\to [0,1],\quad \text{with }f_{x_i,\dl,\eps}(x)=
\begin{cases}
1 \qquad\text{for }x\in \overline{B_{\dl}(x_i)},\\
0 \qquad\text{for }x\in X\setminus B_{\dl+\eps}(x_i),
\end{cases}
\ee
where the closed subsets $\overline{B_{\dl}(x_i)}$ and $X\setminus B_{\dl+\eps}(x_i)$ of $X$ are disjoint by definition.
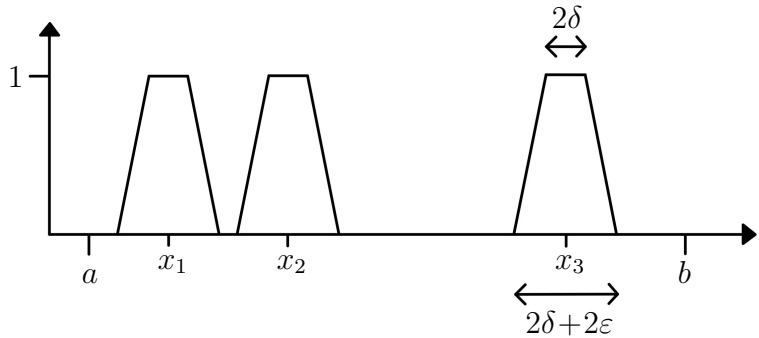
\begin{figure}[!t]
\centering
\captionsetup{justification=centering,margin=1cm}
\begin{center}
\begin{overpic}[width=0.7\textwidth,tics=10]{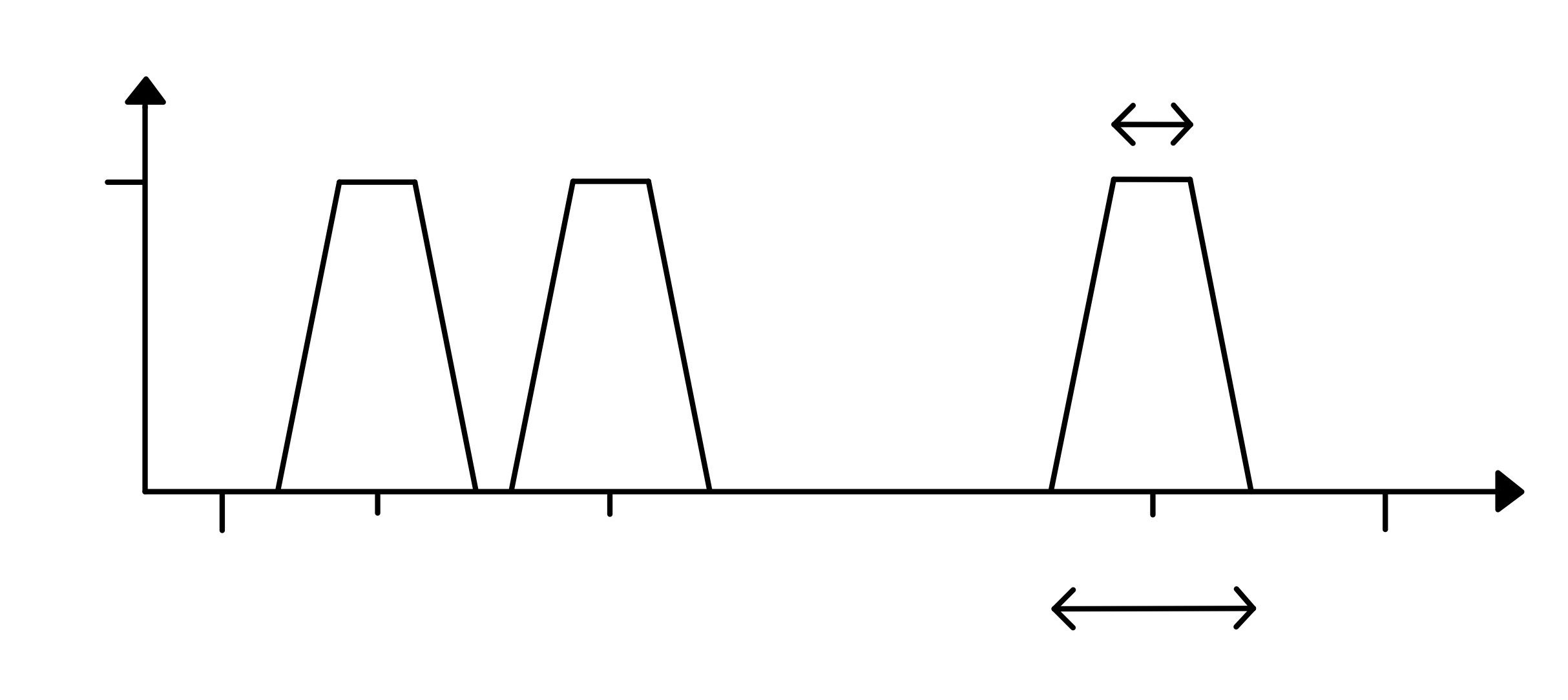}
 \put (4,31.7) {$1$}
 \put (13.2,7.2) {$a$}
 \put (22.7,9) {$x_1$}
 \put (37.5,9) {$x_2$}
 \put (72.2,9) {$x_3$}
 \put (87.5,7.2) {$b$}
 \put (74,2.2) {\makebox(0,0){$2\dl\!+\!2\eps$}}
 \put (71.8,39) {$2\dl$}
\end{overpic}
\caption{\label{figure:functions}
Functions $f_{x_i,\dl,\eps}$ corresponding to given points $x_i\in X$ for $i\in\{1,\dots,n\}$ and given $\dl,\eps>0$ in the case where $X=[a,b]$ and $n=3$.}
\end{center}
\end{figure}

Motivated by \cite{mercer1909functions}, we introduce several types of subsets of $X\times X$, which are illustrated in Figure \ref{figure:square}. For all $i,j=1,\dots ,n$ let  
\be \label{eq:subsets of XxX}
Q_{ij}\vcentcolon=\overline{B_{\dl+\eps}(x_i)} \times \overline{B_{\dl+\eps}(x_j)},\quad 
q_{ij}\vcentcolon=\overline{B_{\dl}(x_i)}\times \overline{B_{\dl}(x_j)},\quad
r_{ij}\vcentcolon =Q_{ij}\setminus q_{ij}.
\ee
Now let $c_1,\dots, c_n\in\R^N$ be arbitrary vectors. Recalling \eqref{eq:cutoff function}, define the continuous and integrable function,
\be\label{eq:definition f}
f_{\dl,\eps} :X\to\R^N,\quad f_{\dl,\eps}(x)\vcentcolon =\sum_{i=1}^n c_i \,\frac{1}{\mu(\overline{B_{\dl}(x_i)})}f_{x_i,\dl,\eps}(x),
\ee
where $0<\smash{\mu(\overline{B_{\dl}(x_i)})<\infty}$ due to the assumption of Theorem \ref{thm:generalization} that the support of $\mu$ is $X$ and \eqref{eq:finite measure}.

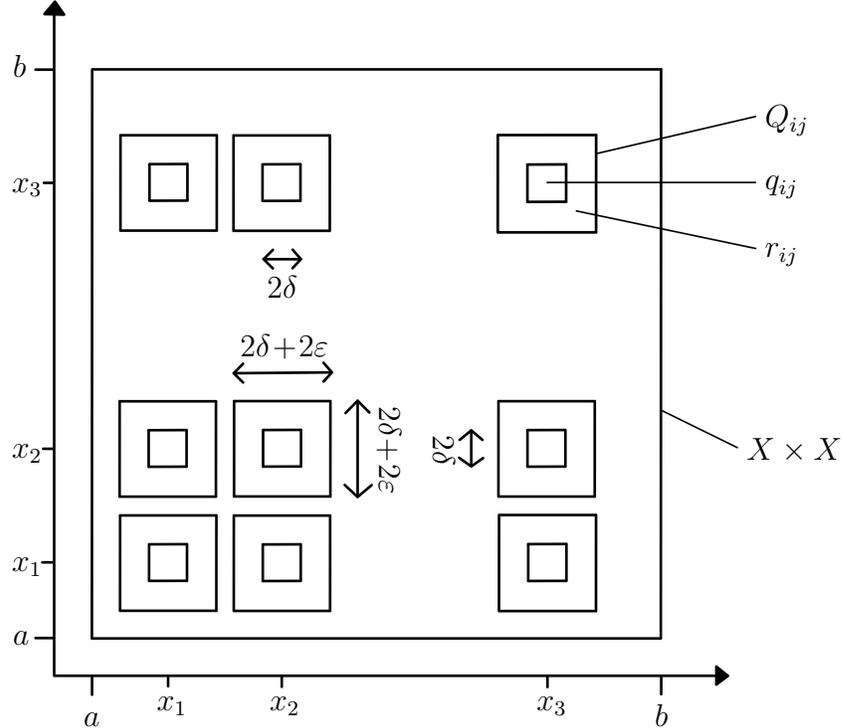
\begin{figure}[!t]
\centering
\captionsetup{justification=centering,margin=0.7cm}
\begin{center}
\begin{overpic}[width=0.8\textwidth,tics=10]{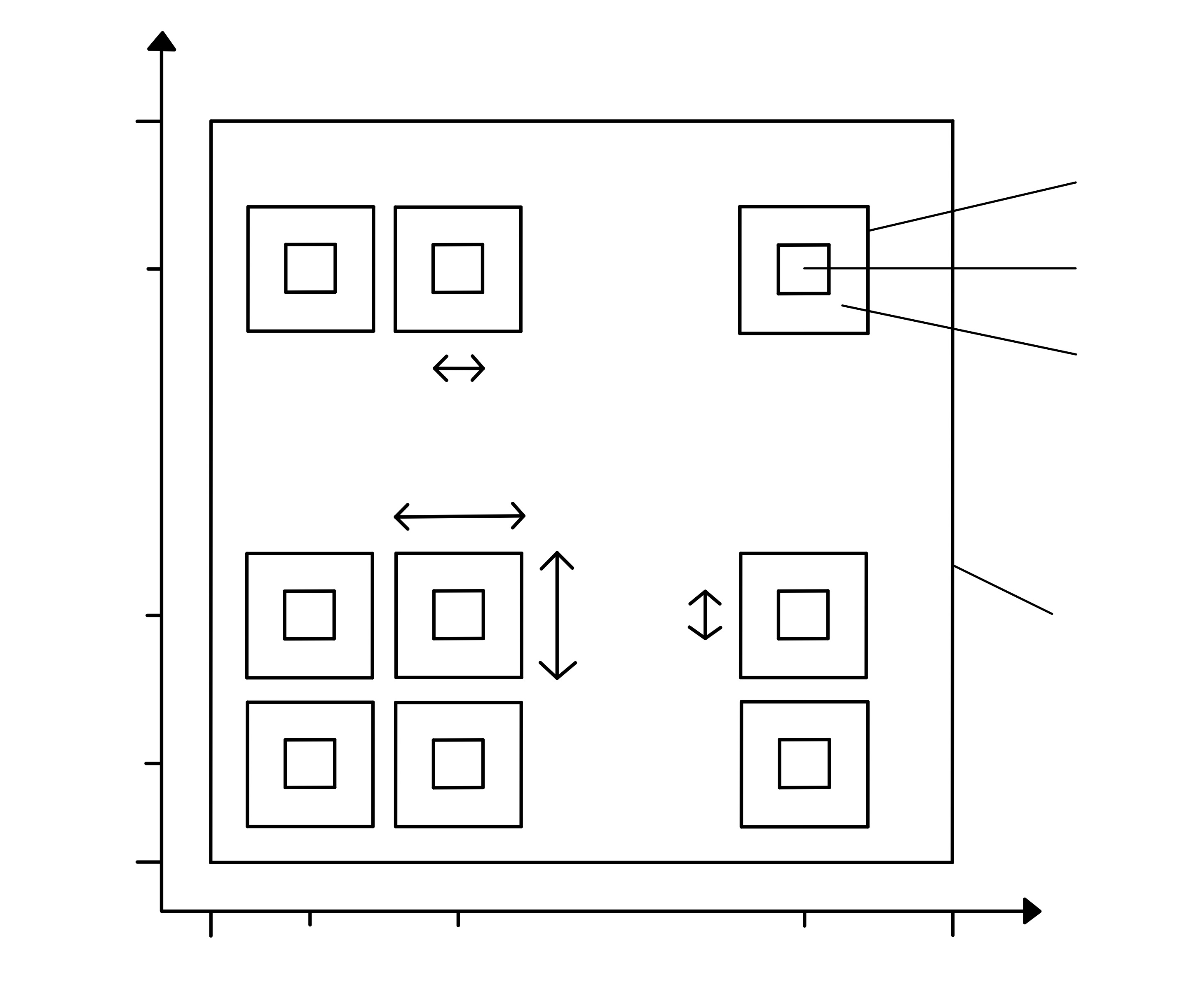}
\put (16.7,2.6){$a$}
\put (24.7,4.3){$x_1$}
\put (37,4.3){$x_2$}
\put (66.1,4.3){$x_3$}
\put (79,2.6){$b$}
\put (9,11.5) {$a$}
\put (8.9,19.5) {$x_1$}
\put (8.9,31.9) {$x_2$}
\put (8.9,60.9) {$x_3$}
\put (9,73.4) {$b$}
\put (36.7,49.4) {$2\dl$}
\put(38.6,44){\makebox(0,0){$2\dl\!+\!2\eps$}}
\put(50,32.7){\makebox(0,0){\rotatebox{-90}{$2\dl\!+\!2\eps$}}}
\put(56,32.7){\makebox(0,0){\rotatebox{-90}{$2\dl$}}}
\put(91,68){$Q_{ij}$}
\put(91,61.2){$q_{ij}$}
\put(91,53.8){$r_{ij}$}
\put(89,31.7){$X\times X$}
\end{overpic}
\caption{\label{figure:square}
The subsets $Q_{ij}$, $q_{ij}$, $r_{ij}\subset X\times X$ associated with given points $x_i,x_j\in X$ for $i,j\in\{1,\dots,n\}$ in the case where $X=[a,b]$ and $n=3$.}
\end{center}
\end{figure}
We denote by $\mu^2=\mu\times\mu$ the product measure on $\mathcal{B}(X\times X)$ and by  $\|\cdot\|$ the Frobenius norm. It follows from  \eqref{eq:disjoint balls}, \eqref{eq:cutoff function}, \eqref{eq:subsets of XxX}, \eqref{eq:definition f}, and the submultiplicativity of the Frobenius norm that, 
\be \label{eq:values of fKf}
\begin{aligned}
&f_{\dl,\eps}(x)^\top K(x,y)f_{\dl,\eps}(y)=\frac{1}{\mu^2(q_{ij})}c_i^\top K(x,y)c_j,  && \text{for }(x,y)\in q_{ij},\\
&|f_{\dl,\eps}(x)^\top K(x,y)f_{\dl,\eps}(y)|\leq \frac{\|c_i\|\|c_j\|}{\mu^2(q_{ij})}\sup\limits_{x,y\in X}  \|K(x,y)\|, && \text{for }(x,y)\in r_{ij},\\
&f_{\dl,\eps}(x)^\top K(x,y)f_{\dl,\eps}(y)=0, && \text{for }(x,y)\notin {\textstyle \bigcup\limits_{i,j=1}^n} Q_{ij} . 
\end{aligned}
\ee
Therefore due to \eqref{eq:values of fKf}, 
\begin{align}\label{eq:seperation of integral}
&\int_X\int_X f_{\dl,\eps}(x)^\top K(x,y)f_{\dl,\eps}(y)d\mu(x)d\mu(y)\\
&=\sum_{i,j=1}^n\frac{1}{\mu^2(q_{ij})}\int_{q_{ij}}c_i^\top K(x,y)c_jd\mu^2(x,y)
+\sum_{i,j=1}^n\int_{r_{ij}}f_{\dl,\eps}(x)^\top K(x,y)f_{\dl,\eps}(y)d\mu^2(x,y). 
\end{align}
Using \eqref{eq:values of fKf}, the second term on the right-hand side of \eqref{eq:seperation of integral} can be bounded as follows,
\be\label{eq:bound r_ij}
\Big|\sum_{i,j=1}^n\int_{r_{ij}}f_{\dl,\eps}(x)^\top K(x,y)f_{\dl,\eps}(y)d\mu^2(x,y)\Big|
\leq \sum_{i,j=1}^n\frac{\mu^2(r_{ij})}{\mu^2(q_{ij})}\|c_i\|\|c_j\|\sup\limits_{x,y\in X}  \|K(x,y)\|.
\ee
We prove \eqref{eq:ipd wrt C}$\,\Rightarrow\,$\eqref{eq:positive definite kernel real} by contradiction. Assume \eqref{eq:ipd wrt C}. Then if \eqref{eq:positive definite kernel real} does not hold, there exists a real number $C>0$, depending on $x_1,\dots,x_n$ and $c_1,\dots,c_n$, with 
\be\label{eq:sum negative}
\sum_{i,j=1}^n c_i^\top K(x_i,x_j)c_j\leq -C.
\ee
We construct the functions $f_{\dl,\eps}$ in \eqref{eq:definition f} as follows. 
Due to the continuity of $K$, we can choose $\dl>0$ and thereby the sets $q_{ij}$ sufficiently small such that 
\be\label{eq:bound cKc}
\big|c_i^\top K(x,y)c_j-c_i^\top K(x_i,x_j)c_j\big|\leq \frac{C}{2n^2},\quad\text{for all }(x,y)\in q_{ij},\ i,j=1,\dots,n.
\ee
Note that $\eps>0$ in $f_{\dl,\eps}$ will be fixed later.

It follows from \eqref{eq:sum negative} and \eqref{eq:bound cKc} that the first term on the right-hand side of \eqref{eq:seperation of integral} is bounded from above as follows,
\be
\begin{aligned}\label{eq:bound -C/2}
&\sum_{i,j=1}^n\frac{1}{\mu^2(q_{ij})}\int_{q_{ij}}c_i^\top K(x,y)c_jd\mu^2(x,y)\\
&\leq \sum_{i,j=1}^n\frac{1}{\mu^2(q_{ij})}\int_{q_{ij}}\Big(c_i^\top K(x_i,x_j)c_j+\frac{C}{2n^2}\Big)d\mu^2(x,y)\\
&=\sum_{i,j=1}^n \Big(c_i^\top K(x_i,x_j)c_j+\frac{C}{2n^2}\Big)\\
&\leq -C+\frac{C}{2}=-\frac{C}{2}.
\end{aligned}
\ee
Now \eqref{eq:seperation of integral} and \eqref{eq:bound -C/2} yield the following,
\be
\int_X\int_X f_{\dl,\eps}(x)^\top K(x,y)f_{\dl,\eps}(y)d\mu(x)d\mu(y)
\leq \sum_{i,j=1}^n\int_{r_{ij}}f(x)^\top K(x,y)f_{\dl,\eps}(y)d\mu^2(x,y)-\frac{C}{2},
\ee
which together with assumption \eqref{eq:ipd wrt C} applied to $ f_{\dl,\eps}$ and \eqref{eq:bound r_ij} implies,
\be\label{eq:contradiction inequality}
\frac{C}{2}\leq \sum_{i,j=1}^n\frac{\mu^2(r_{ij})}{\mu^2(q_{ij})}\|c_i\|\|c_j\|\sup\limits_{x,y\in X}  \|K(x,y)\|.
\ee
Note that the right-hand side of \eqref{eq:contradiction inequality} is finite as $K$ is bounded by assumption. Finally, it follows from \eqref{eq:finite measure} and the definition of the sets $r_{ij}$ in \eqref{eq:subsets of XxX} that
\be
\lim_{\eps\to 0}\mu^2(r_{ij})=\lim_{\eps\to 0}\mu^2(\overline{B_{\dl+\eps}(x_i)}\times \overline{B_{\dl+\eps}(x_j)})
- \mu^2(\overline{B_{\dl}(x_i)}\times \overline{B_{\dl}(x_j)})=0,
\ee
which implies that if $\eps>0$ is chosen sufficiently small, \eqref{eq:contradiction inequality} does not hold and we get a contradiction. This finishes the proof of the implication \eqref{eq:ipd wrt C}$\,\Rightarrow\,$\eqref{eq:positive definite kernel real}.

For the third implication \eqref{eq:positive definite kernel real}$\,\Rightarrow\,$\eqref{eq:ipd wrt L1}, assume that \eqref{eq:positive definite kernel real} holds.

\textbf{Case 1}: First assume that $\mu$ is a finite measure.
The key idea of the proof is to apply a generalized version of Mercer's theorem for matrix-valued kernels (see \cite{devito2013extension}). In order to apply this result, note that \eqref{eq:positive definite kernel real} implies \eqref{eq:positive definite kernel complex} by Lemma \ref{lemma:equivalence of real and complex}. Moreover, since $K$ is bounded and $\mu$ is finite by assumption, it holds that 
\be \label{eq:bound tr}
\int_X \operatorname{Tr} K(x,x)d\mu(x)\leq N\mu(X) \sup_{\\
1\leq \ell \leq N}\sup_{x\in X}  K_{\ell \ell}(x,x)   <\infty,
\ee
where Tr denotes the trace of a matrix in $\R^{N\times N}$. Note that $G$ must be nonnegative on the diagonal, which can be seen from condition \eqref{eq:positive definite kernel real} by choosing $n=1$ and $c_1=e_l$ for example. 
Next, define the linear integral operator
\be
\mathbf{K}:L^2(X,\mathbb{C}^N)\to L^2(X,\mathbb{C}^N), \quad (\mathbf{K}f)(x)\vcentcolon= \int_X K(x,y)f(y)d\mu(y),
\ee
which is well-defined by Theorem A.1 in \cite{devito2013extension}, and denote by $\operatorname{ker}(\mathbf{K})$ its kernel.
Since \eqref{eq:positive definite kernel complex} and \eqref{eq:bound tr} hold, we can apply the generalized version of Mercer's theorem for matrix-valued kernels (see Theorem 4.1 in \cite{devito2013extension}). In combination with the auxiliary Theorem A.1 in \cite{devito2013extension}, it implies the existence of a countable orthonormal basis $\{\varphi_k\}_{k\in I}$ of $\text{ker}(\smash{\mathbf{K}})^\perp\subset L^2(X,\mathbb{C}^N)$ of continuous eigenfunctions of $\smash{\mathbf{K}}$ with a corresponding family $\{\sigma_k\}_{k\in I}\subset(0,\infty)$ of positive eigenvalues such that
\be \label{eq:expansion}
K_{ \ell m}(x,y)=\sum_{k\in I}\sigma_k \overline{\varphi_k^\ell  (x)}\varphi_k^m(y),\quad \textrm{for all } (x,y)\in X\times X, \ \ell,m\in\{1,\dots,N\},
\ee
where the series converges uniformly on $X\times X$. Now let $f\in L^1(X,\R^N)$. From \eqref{eq:expansion} it follows that, 
\be \label{eq:expansion integral}
\begin{aligned}
&\int_X\int_X f(x)^\top K(x,y)f(y)d\mu (x)d\mu (y)\\
&=\int_X\int_X \sum_{\ell,m=1}^N f^\ell(x) \Big(\sum_{k\in I}\sigma_k \overline{\varphi_k^\ell(x)}\varphi_k^m(y)\Big) f^m(y)d\mu(x)d\mu(y) \\ 
&=\int_X\int_X \sum_{k\in I}\sigma_k \sum_{\ell,m=1}^N f^\ell(x) \overline{\varphi_k^\ell(x)}\varphi_k^m(y) f^m(y)d\mu(x)d\mu(y) \\
&=\int_X\int_X \int_{I} \sigma_k \sum_{\ell,m=1}^N f^\ell(x) \overline{\varphi_k^\ell(x)}\varphi_k^m(y) f^m(y)d\nu (k)d\mu(x)d\mu(y),
\end{aligned} 
\ee
where $\nu$ is the counting measure on the countable index set of the eigenfunctions $I$. In particular, $\nu$ is $\sigma$-finite. 

Next, we would like to switch the order of integration on the right-hand side of \eqref{eq:expansion integral} using Fubini's theorem. In order to do that, we show that the following integral is finite, using Young's inequality, \eqref{eq:bound tr}, \eqref{eq:expansion} and the fact that $f\in L^1(X,\R^N)$,
\bn
&& \int_X\int_X \int_{I} \big|\sigma_k \sum_{\ell,m=1}^N f^\ell(x) \overline{\varphi_k^\ell(x)}\varphi_k^m(y) f^m(y)\big| d\nu (k)d\mu(x)d\mu(y) \\
&&
\leq \int_X\int_X \sum_{\ell,m=1}^N \big|f^\ell(x) \big|   \int_{I} \sigma_k \big(\big|\varphi_k^\ell(x)\big|^2+\big|\varphi_k^m(y) \big|^2\big) d\nu (k)\big|f^m(y)\big|d\mu(x)d\mu(y) \\
&&= \int_X\int_X \sum_{\ell,m=1}^N \big|f^\ell(x) \big| \big(K_{\ell \ell }(x,x)+K_{mm}(y,y)\big)  \big|f^m(y)\big|d\mu(x)d\mu(y) \\ 
&&
\leq  2\Big(\sup_{1\leq \ell \leq N} \sup_{x\in X}K_{\ell \ell}(x,x)\Big) \sum_{\ell,m=1}^N\int_X \big|f^\ell(x) \big|d\mu(x) \int_X \big|f^m(y)\big| d\mu(y) \\ 
&&<\infty.
\en
Let $\mathbf{1}_N\in\R^{N\times N}$ denote the matrix of ones in all entries, which is symmetric positive semidefinite as its eigenvalues are $N$ with multiplicity $1$ and $0$ with multiplicity $N-1$. 
Moreover, for every $k\in I$ we define 
\be
g_k:X\to\mathbb{C}^N,\quad g_k(x)\vcentcolon
=\begin{pmatrix} \varphi_k^1(x) f^1(x) \\ \vdots \\ \varphi_k^N(x) f^N(x) \end{pmatrix} \;,
\ee
which is the Hadamard product of $\varphi_k$ and $f$.
Then \eqref{eq:expansion integral} and Fubini's theorem imply 
\be
\begin{aligned}
&\int_X\int_X f(x)^\top K(x,y)f(y)d\mu (x)d\mu (y)\\
&=\int_{I}^{\ } \int_X\int_X  \sigma_k \sum_{\ell,m=1}^N f^\ell(x) \overline{\varphi_k^\ell(x)}\varphi_k^m(y) f^m(y)d\mu(x)d\mu(y)d\nu (k) \\
&=\int_{I}^{\ }\int_X\int_X \sigma_k\ \overline{g_k(x)}^\top \mathbf{1}_N \ g_k(y) d\mu(x)d\mu(y)d\nu (k) \\ 
&=\sum_{k\in I} \sigma_k \Big(\overline{\int_X  g_k(x)d\mu(x)}\Big)^\top  \mathbf{1}_N\  \Big(\int_X g_k(y) d\mu(y)\Big)\geq 0.
\end{aligned}
\ee
This shows the implication \eqref{eq:positive definite kernel real}$\,\Rightarrow\,$\eqref{eq:ipd wrt L1} in the case where $\mu$ is finite.

\textbf{Case 2}: Now assume that $\mu$ is a locally finite measure. 
Then, for every $x\in X$ there exists an open neighborhood $U_x\subset X$ of $x$ with $\mu(U_x)<\infty$. As a separable metric space, $X$ is also a Lindelöf space, that is, every open cover of $X$ has a countable subcover. In particular, there exists a countable subcover of $(U_x)_{x\in X}$, which yields that $\mu$ is $\sigma$-finite.

Let $(X_s)_{s\in\N}$ be an increasing sequence of $\mathcal{B}(X)$-measurable subsets of $X$ with $\mu(X_s)<\infty$ for all $s\geq 1$ and $\bigcup_{s\geq 1}X_s =X$. Let $f\in L^1(X,\R^N)$. Define the sequence of functions $(F_s)_{s\geq 1}$ by
\be
F_s:X\times X\to\R,\quad F_s(x,y)\vcentcolon= \one_{X_s}(x)\one_{X_s}(y)f(x)^\top K(x,y)f(y),\quad \textrm{for } s\geq 1.
\ee
Then $(F_s)_{s\geq 1}$ converges pointwise to the function $f(x)^\top K(x,y)f(y)$ on $X\times X$, since for every $(x,y)\in X\times X$ there exists an $s_0\in\N$ such that $x,y\in X_s$ for all $s\geq s_0$. Moreover, it holds for all $s\geq 1$ that 
\be\begin{aligned}
&\int_X\int_X\big|F_s(x,y)\big|d\mu(x)d\mu(y)\\
&\leq \int_X\int_X\one_{X_s}(x)\one_{X_s}(y)\big|f(x)\big|\big\|K(x,y)\big\|\big|f(y)\big|d\mu(x)d\mu(y)  \\
&\leq \int_X\int_X\big|f(x)\big|\big|f(y)\big|d\mu(x)d\mu(y)\sup\limits_{x,y\in X}  \|K(x,y)\|\\
&<\infty,
\end{aligned}\ee
where we used the submultiplicativity of the Frobenius norm, the integrability of $f$ and the boundedness of $K$ by the hypothesis. Thus dominated convergence can be applied to get
\be\begin{aligned}
\int_X\int_X f(x)^\top K(x,y)f(y)d\mu (x)d\mu (y)
&=\lim_{s\to\infty}\int_X\int_X F_s(x,y)d\mu (x)d\mu (y)\\
&=\lim_{s\to\infty}\int_{X_s}\int_{X_s} f(x)^\top K(x,y)f(y)d\mu (x)d\mu (y)\\
&\geq 0,
\end{aligned}\ee
where the last inequality follows from Case 1, since for all $s\geq 1$ we have $\mu(X_s)<\infty$ and that $(X_s,\mathcal{B}(X_s),\mu|_{X_s})$ and $K|_{X_s\times X_s}$ satisfy the assumptions of Theorem \ref{thm:generalization} and condition \eqref{eq:positive definite kernel real}.

This shows the implication \eqref{eq:positive definite kernel real}$\,\Rightarrow\,$\eqref{eq:ipd wrt L1} in the case where $\mu$ is locally finite and completes the proof.

\end{proof}


\begin{thebibliography}{45}
\providecommand{\natexlab}[1]{#1}
\providecommand{\url}[1]{\texttt{#1}}
\expandafter\ifx\csname urlstyle\endcsname\relax
  \providecommand{\doi}[1]{doi: #1}\else
  \providecommand{\doi}{doi: \begingroup \urlstyle{rm}\Url}\fi

\bibitem[Abi~Jaber et~al.(2021)Abi~Jaber, Miller, and Pham]{abi2021integral}
E.~Abi~Jaber, E.~Miller, and H.~Pham.
\newblock Integral operator {R}iccati equations arising in stochastic
  {V}olterra control problems.
\newblock \emph{SIAM Journal on Control and Optimization}, 59\penalty0
  (2):\penalty0 1581--1603, 2021.

\bibitem[Abi~Jaber et~al.(2023)Abi~Jaber, Neuman, and
  Vo{\ss}]{abijaber2023equilibrium}
E.~Abi~Jaber, E.~Neuman, and M.~Vo{\ss}.
\newblock Equilibrium in functional stochastic games with mean-field
  interaction.
\newblock \emph{arXiv preprint arXiv:2306.05433}, 2023.

\bibitem[Alfonsi and Schied(2013)]{Alf-Schied-13}
A.~Alfonsi and A.~Schied.
\newblock Capacitary measures for completely monotone kernels via singular
  control.
\newblock \emph{SIAM Journal on Control and Optimization}, 51\penalty0
  (2):\penalty0 1758--1780, 2013.

\bibitem[Alfonsi et~al.(2016)Alfonsi, Kl{\"o}ck, and
  Schied]{alfonsi2016multivariate}
A.~Alfonsi, F.~Kl{\"o}ck, and A.~Schied.
\newblock Multivariate transient price impact and matrix-valued positive
  definite functions.
\newblock \emph{Mathematics of Operations Research}, 41\penalty0 (3):\penalty0
  914--934, 2016.

\bibitem[Aronszajn(1950)]{aronszajn1950theory}
N.~Aronszajn.
\newblock Theory of reproducing kernels.
\newblock \emph{Transactions of the American Mathematical Society}, 68\penalty0
  (3):\penalty0 337--404, 1950.

\bibitem[Benatia et~al.(2017)Benatia, Carrasco, and Florens]{BENATIA2017269}
D.~Benatia, M.~Carrasco, and J.~Florens.
\newblock Functional linear regression with functional response.
\newblock \emph{Journal of Econometrics}, 201\penalty0 (2):\penalty0 269--291,
  2017.

\bibitem[Bhatia(2009)]{bhatia2009positive}
R.~Bhatia.
\newblock \emph{Positive definite matrices}.
\newblock Princeton University Press, 2009.

\bibitem[Bj{\"o}rck(1956)]{Bjorck56}
G.~Bj{\"o}rck.
\newblock Distributions of positive mass, which maximize a certain generalized
  energy integral.
\newblock \emph{Arkiv f{\"o}r Matematik}, 3\penalty0 (3):\penalty0 255--269,
  1956.

\bibitem[Bochner(1932)]{bochner1932vorlesungen}
S.~Bochner.
\newblock \emph{Vorlesungen uber Fouriersche Integrale}.
\newblock Akademische Verlagsgesellschaft, 1932.

\bibitem[Bochner(1933)]{bochner1933monotone}
S.~Bochner.
\newblock Monotone {F}unktionen, {S}tieltjessche {I}ntegrale und harmonische
  {A}nalyse.
\newblock \emph{Mathematische Annalen}, 108\penalty0 (1):\penalty0 378--410,
  1933.

\bibitem[Borodachov et~al.(2019)Borodachov, Hardin, and Saff]{Borodachov19}
S.~Borodachov, D.~Hardin, and E.~Saff.
\newblock \emph{Discrete Energy on Rectifiable Sets}.
\newblock Springer Monographs in Mathematics, 2019.

\bibitem[Brouard et~al.(2016)Brouard, Szafranski, and d'Alch{\'e}
  Buc]{brouard2016input}
C.~Brouard, M.~Szafranski, and F.~d'Alch{\'e} Buc.
\newblock Input output kernel regression: Supervised and semi-supervised
  structured output prediction with operator-valued kernels.
\newblock \emph{Journal of Machine Learning Research}, 17\penalty0
  (176):\penalty0 1--48, 2016.

\bibitem[Buhmann(2000)]{buhmann2000radial}
M.~D. Buhmann.
\newblock Radial basis functions.
\newblock \emph{Acta Numerica}, 9:\penalty0 1--38, 2000.

\bibitem[Caponnetto et~al.(2008)Caponnetto, Micchelli, Pontil, and
  Ying]{caponnetto2008universal}
A.~Caponnetto, C.~A. Micchelli, M.~Pontil, and Y.~Ying.
\newblock Universal multi-task kernels.
\newblock \emph{The Journal of Machine Learning Research}, 9:\penalty0
  1615--1646, 2008.

\bibitem[Carath{\'e}odory(1907)]{caratheodory1907variabilitatsbereich}
C.~Carath{\'e}odory.
\newblock {\"U}ber den {V}ariabilit{\"a}tsbereich der {K}oeffizienten von
  {P}otenzreihen, die gegebene {W}erte nicht annehmen.
\newblock \emph{Mathematische Annalen}, 64\penalty0 (1):\penalty0 95--115,
  1907.

\bibitem[Cooper(1960)]{cooper1960positive}
J.~Cooper.
\newblock Positive definite functions of a real variable.
\newblock \emph{Proceedings of the London Mathematical Society}, 3\penalty0
  (1):\penalty0 53--66, 1960.

\bibitem[De~Vito et~al.(2013)De~Vito, Umanit{\`a}, and
  Villa]{devito2013extension}
E.~De~Vito, V.~Umanit{\`a}, and S.~Villa.
\newblock {An extension of Mercer theorem to matrix-valued measurable kernels}.
\newblock \emph{Applied and Computational Harmonic Analysis}, 34\penalty0
  (3):\penalty0 339--351, 2013.

\bibitem[Fasshauer(2011)]{fasshauer2011positive}
G.~E. Fasshauer.
\newblock Positive definite kernels: past, present and future.
\newblock \emph{Dolomites Research Notes on Approximation}, 4:\penalty0 21--63,
  2011.

\bibitem[Ferreira and Menegatto(2009)]{ferreira2009eigenvalues}
J.~Ferreira and V.~Menegatto.
\newblock Eigenvalues of integral operators defined by smooth positive definite
  kernels.
\newblock \emph{Integral Equations and Operator Theory}, 64\penalty0
  (1):\penalty0 61--81, 2009.

\bibitem[Fuglede(1960)]{Fuglede60}
B.~Fuglede.
\newblock {On the theory of potentials in locally compact spaces}.
\newblock \emph{Acta Mathematica}, 103\penalty0 (3-4):\penalty0 139--215, 1960.

\bibitem[Gatheral et~al.(2012)Gatheral, Schied, and Slynko]{GSS}
J.~Gatheral, A.~Schied, and A.~Slynko.
\newblock Transient linear price impact and {F}redholm integral equations.
\newblock \emph{Mathematical Finance}, 22:\penalty0 445--474, 2012.

\bibitem[Giesl(2021)]{giesl2021matrix}
P.~Giesl.
\newblock On a matrix-valued pde characterizing a contraction metric for a
  periodic orbit.
\newblock \emph{Discrete and Continuous Dynamical Systems-B}, 26\penalty0
  (9):\penalty0 4839--4865, 2021.

\bibitem[Giesl and Wendland(2018)]{giesl2018kernel}
P.~Giesl and H.~Wendland.
\newblock Kernel-based discretization for solving matrix-valued {PDE}s.
\newblock \emph{SIAM Journal on Numerical Analysis}, 56\penalty0 (6):\penalty0
  3386--3406, 2018.

\bibitem[Giesl et~al.(2021)Giesl, Hafstein, and
  Mehrabinezhad]{giesl2021computation}
P.~Giesl, S.~Hafstein, and I.~Mehrabinezhad.
\newblock Computation and verification of contraction metrics for exponentially
  stable equilibria.
\newblock \emph{Journal of Computational and Applied Mathematics},
  390:\penalty0 113332, 2021.

\bibitem[Gl{\"o}ckner(2003)]{glockner2003positive}
H.~Gl{\"o}ckner.
\newblock \emph{Positive definite functions on infinite-dimensional convex
  cones}.
\newblock American Mathematical Society, 2003.

\bibitem[Hardin et~al.(2004)Hardin, Saff, et~al.]{Hardin04}
D.~P. Hardin, E.~B. Saff, et~al.
\newblock Discretizing manifolds via minimum energy points.
\newblock \emph{Notices of the American Mathematical Society}, 51\penalty0
  (10):\penalty0 1186--1194, 2004.

\bibitem[Helms(2009)]{Helms09}
L.~L. Helms.
\newblock \emph{Potential theory}.
\newblock Universitext. Springer London, 2009.

\bibitem[Hilbert(1904)]{hilbert1904grundzuge}
D.~Hilbert.
\newblock {Grundz{\"u}ge einer allgemeinen Theorie der linearen
  Integralgleichungen I}.
\newblock \emph{G{\"o}ttinger Nachrichten}, 1904:\penalty0 49--91, 1904.

\bibitem[Kadri et~al.(2016)Kadri, Duflos, Preux, Canu, Rakotomamonjy, and
  Audiffren]{kadri2016operator}
H.~Kadri, E.~Duflos, P.~Preux, S.~Canu, A.~Rakotomamonjy, and J.~Audiffren.
\newblock Operator-valued kernels for learning from functional response data.
\newblock \emph{Journal of Machine Learning Research}, 17\penalty0
  (20):\penalty0 1--54, 2016.

\bibitem[Mathias(1923)]{mathias1923positive}
M.~Mathias.
\newblock {\"U}ber positive {F}ourier-{I}ntegrale.
\newblock \emph{Mathematische Zeitschrift}, 16\penalty0 (1):\penalty0 103--125,
  1923.

\bibitem[Mercer(1909)]{mercer1909functions}
J.~Mercer.
\newblock Functions of positive and negative type, and their connection the
  theory of integral equations.
\newblock \emph{Philosophical Transactions of the Royal Society of London.
  Series A}, 209:\penalty0 415--446, 1909.

\bibitem[Minh et~al.(2016)Minh, Bazzani, and Murino]{minh2016unifying}
H.~Q. Minh, L.~Bazzani, and V.~Murino.
\newblock A unifying framework in vector-valued reproducing kernel hilbert
  spaces for manifold regularization and co-regularized multi-view learning.
\newblock \emph{Journal of Machine Learning Research}, 17\penalty0
  (25):\penalty0 1--72, 2016.

\bibitem[Moore(1916)]{moore1916properly}
E.~H. Moore.
\newblock On properly positive {H}ermitian matrices.
\newblock \emph{Bulletin of the American Mathematical Society}, 23, 1916.

\bibitem[Moore(1935)]{moore1935general}
E.~H. Moore.
\newblock General analysis {I}.
\newblock \emph{American Philosophical Society}, 1935.

\bibitem[Neuman and Zhang(2025)]{zhang_24}
E.~Neuman and Y.~Zhang.
\newblock Statistical learning with sublinear regret of propagator models.
\newblock \emph{to appear in Annals of Applied Probability}, 2025.

\bibitem[Neuman et~al.(2023)Neuman, Stockinger, and Zhang]{Wolf_24}
E.~Neuman, W.~Stockinger, and Y.~Zhang.
\newblock An offline learning approach to propagator models.
\newblock \emph{arXiv:2309.02994}, 2023.

\bibitem[Rudin(2017)]{rudin2017fourier}
W.~Rudin.
\newblock \emph{Fourier analysis on groups}.
\newblock Courier Dover Publications, 2017.

\bibitem[Saha and Palaniappan(2020)]{saha2020learning}
A.~Saha and B.~Palaniappan.
\newblock Learning with operator-valued kernels in reproducing kernel {K}rein
  spaces.
\newblock \emph{Advances in Neural Information Processing Systems},
  33:\penalty0 13856--13866, 2020.

\bibitem[Steinwart and Scovel(2012)]{steinwart2012mercer}
I.~Steinwart and C.~Scovel.
\newblock Mercer's theorem on general domains: On the interaction between
  measures, kernels, and {RKHS}s.
\newblock \emph{Constructive Approximation}, 35:\penalty0 363--417, 2012.

\bibitem[Stewart(1976)]{stewart1976positive}
J.~Stewart.
\newblock Positive definite functions and generalizations, an historical
  survey.
\newblock \emph{The Rocky Mountain Journal of Mathematics}, 6\penalty0
  (3):\penalty0 409--434, 1976.

\bibitem[Sun(2005)]{sun2005mercer}
H.~Sun.
\newblock Mercer theorem for {RKHS} on noncompact sets.
\newblock \emph{Journal of Complexity}, 21\penalty0 (3):\penalty0 337--349,
  2005.

\bibitem[Toeplitz(1911)]{toeplitz1911fourier}
O.~Toeplitz.
\newblock {\"U}ber die {F}ourier'sche {E}ntwickelung positiver {F}unktionen.
\newblock \emph{Rendiconti del Circolo Matematico di Palermo}, 32\penalty0
  (1):\penalty0 191--192, 1911.

\bibitem[Wang et~al.(2023)Wang, Yong, and Zhou]{wang2022linear}
H.~Wang, J.~Yong, and C.~Zhou.
\newblock Linear-quadratic optimal controls for stochastic {V}olterra integral
  equations: Causal state feedback and path-dependent {R}iccati equations.
\newblock \emph{SIAM Journal on Control and Optimization}, 61\penalty0
  (4):\penalty0 2595--2629, 2023.

\bibitem[Willard(2012)]{willard2012general}
S.~Willard.
\newblock \emph{General Topology}.
\newblock Dover Publications, 2012.

\bibitem[Young(1910)]{young1910note}
W.~H. Young.
\newblock A note on a class of symmetric functions and on a theorem required in
  the theory of integral equations.
\newblock \emph{Messenger of Mathematics}, 40:\penalty0 37--43, 1910.

\end{thebibliography}

\end{document}